\documentclass[10pt,reqno]{amsart}
  \usepackage{geometry}
\usepackage{color, graphicx}
\usepackage[usenames,dvipsnames,svgnames,table]{xcolor}
\usepackage{tikz} 

\usepackage[T1]{fontenc}
\usepackage[english]{babel}
\providecommand{\U}[1]{\protect\rule{.1in}{.1in}}

\newtheorem{theorem}{Theorem}[section]

\newtheorem{corollary}[theorem]{Corollary}

\newtheorem{definition}[theorem]{Definition}
\newtheorem{example}[theorem]{Example}

\newtheorem{lemma}[theorem]{Lemma}
\newtheorem{notation}[theorem]{Notation}

\newtheorem{proposition}[theorem]{Proposition}
\newtheorem{remark}[theorem]{Remark}

\def\N{\mathbb{N}}

\newcommand{\vs}[1]{\langle #1 \rangle}

\newcommand{\arrow}[2]{\overset{#1}{\xrightarrow{\hspace*{#2}}}}

\DeclareMathOperator{\shad}{Shad}
\DeclareMathOperator{\gaps}{gaps}
\DeclareMathOperator{\cogaps}{cogaps}
\DeclareMathOperator{\maxgen}{maxgen}
\DeclareMathOperator{\pred}{pred}
\DeclareMathOperator{\suc}{succ}
\DeclareMathOperator{\pre}{pre}

\title{Gotzmann monomials in four variables}

\author{V. Bonanzinga and S. Eliahou}

\begin{document}
\maketitle

\begin{abstract} It is a widely open problem to determine which monomials in the $n$-variable polynomial ring $K[x_1,\dots,x_n]$ over a field $K$ have the Gotzmann property, \emph{i.e.} induce a Borel-stable Gotzmann monomial ideal. Since 2007, only the case $n \le 3$ was known. Here we solve the problem for the case $n=4$. The solution involves a surprisingly intricate characterization.
\end{abstract}

\medskip


\smallskip


\section{Introduction}\label{sec introduction}

Let $K$ be a field and let $R_n=K[x_1,\dots,x_n]$ be the $n$-variable polynomial algebra over $K$ endowed with its usual grading $\deg(x_i)=1$ for all $i$. We denote by $S_n \subset R_n$ the set of all monomials $u=x_1^{a_1}\cdots x_n^{a_n}$ in $R_n$, and by $S_{n,d} \subset S_n$ the subset of monomials of degree $\deg(u)=\sum_i a_i=d$.

A monomial ideal $J \subseteq R_n$ is said to be \emph{Borel-stable} or \emph{strongly stable} if for any monomial $v \in J$ and any variable $x_j$ dividing $v$, one has $x_iv/x_j \in J$ for all $1 \le i \le j$. Given a monomial $u \in S_n$, let $\vs{u}$ denote the smallest Borel-stable monomial ideal in $R_n$ containing $u$, and let $B(u)$ denote the unique minimal system of monomial generators of $\vs{u}$. Then $B(u)$ may be described as the smallest set of monomials containing $u$ and stable under the operations $v \mapsto vx_i/x_j$ whenever $x_j$ divides $v$ and $i \le j$. 

Recall that a homogeneous ideal $I \subseteq R_n$ is a \textit{Gotzmann ideal} if, from a certain degree on, its Hilbert function attains Macaulay's lower bound. See \emph{e.g.} \cite{B,He} for more details. Determining which homogeneous ideals are Gotzmann ideals is notoriously difficult. This will be illustrated in this paper, where our determination of all monomials $u$ in $S_4$ such that the ideal $\vs{u}$ is a Gotzmann ideal involves a surprisingly complicated formula. We introduce the following definition.

\begin{definition} We say that a monomial $u \in S_n$ is a \emph{Gotzmann monomial} if its associated Borel-stable monomial ideal $\vs{u}$ is a Gotzmann ideal. 
\end{definition}

Determining all Gotzmann monomials in $S_n$ is a widely open problem. Indeed, the current knowledge about it is limited to the case $n \le 3$. Specifically, for $n \le 2$ all monomials in $S_1$ or $S_2$ are Gotzmann, whereas for $n=3$, \textit{the monomial $x_1^ax_2^bx_3^t$ is Gotzmann in $S_3$  if and only if $t \ge \binom{b}{2}$.} The latter result can be deduced from \cite[Proposition 8]{Mu2}. A short proof using the general tools developed in this paper will be given in the last section.

\smallskip
The above result for $n=3$ illustrates a general property of Gotzmann monomials, proved in \cite{B} using Gotzmann's persistence theorem.
\begin{theorem}\label{VB 2003} Let $u \in S_n$. 
\begin{enumerate}
\item There exists $k \in \N$ such that $ux_n^k$ is Gotzmann in $S_n$.
\item If $u$ is Gotzmann in $S_n$, then so is $u x_n$.
\end{enumerate}
\end{theorem}

This reduces the determination of Gotzmann monomials in $S_n$ to the following question. Given $u_0 \in S_{n-1}$, what is the least exponent $t \ge 0$ such that $u_0x_n^t$ is a Gotzmann monomial in $S_n$?

\smallskip
Our main result in this paper is the classification of all Gotzmann monomials in $S_4$. It states that \textit{a monomial $u=x_1^{a}x_2^{b}x_3^{c}x_4^{t}$ is a Gotzmann monomial  in $S_4$ if and only if}
$$
t \ \ge \ \binom{\binom{b}{2}}{2}+\frac{b+4}{3}\binom{b}{2}+(b+1)\binom{c+1}{2}+\binom{c+1}{3}-c.
$$
See Theorem~\ref{main thm}. Interestingly, before achieving this rather intricate characterization, all the easy-to-perform computer-algebraic experiments we ran in order to get a clue at it were of no help. Only the conceptual tools developed below allowed us to formulate and prove this result. Completing the analogous task in $S_n$ for $n \ge 5$ remains an open problem.

\subsection{Some related results}

In 2000, Aramova, Avramov and Herzog posed the open problem of determining which monomial ideals are Gotzmann ideals \cite{AAH}. Some partial answers have since emerged. In 2003, the first author characterized all principal Borel ideals with Borel generator up to degree $4$ which are Gotzmann \cite{B}. In 2006, Mermin classified Lexlike ideals, \emph{i.e.} ideals which are generated by initial segments of ``lexlike'' sequences \cite {Me}. In 2007, Murai classified Gotzmann ideals in the polynomial ring in $3$ variables \cite {Mu2}. In 2008, Murai and Hibi described all Gotzmann ideals in $K[x_1,\dots,x_n]$ with fewer than $n$ generators \cite {MH}. In 2008, Loredana Sorrenti, Anda and Oana Olteanu classified Gotzmann ideals which are generated by segments in the lexicographic order \cite {OOS}. In 2012, Hoefel characterized all Gotzmann edge ideals \cite {Ho}. In 2012, Hoefel and Mermin described all Gotzmann squarefree monomial ideals \cite {HM}. See also \cite{PS} for related results.

\subsection{Contents}
In Section~\ref{sec basic}, we recall or introduce basic notions such as lexsegments and lexintervals, the sets of \emph{gaps} and \emph{cogaps} of a monomial, the \emph{maxgen monomial} of a set of monomials, and finally Gotzmann monomials and criteria in terms of gaps and cogaps to recognize them. In Section~\ref{sec gaps}, we focus on properties of the gaps and cogaps of a monomial and how to compute them. In Section~\ref{sec pred and succ}, we describe the lexicographic predecessor and successor of a monomial. Section~\ref{sec maxgen} is devoted to the determination of the maxgen monomial of lexintervals. In Section~\ref{sec x1 xn}, we show some specific behaviors of the first and last variables. Finally, in Section~\ref{sec gotzmann} we use all the material developed in the preceding sections to prove our main theorem on the characterization of Gotzmann monomials in four variables.

\section{Background and basic notions}\label{sec basic}

\subsection{Lexsegments and lexintervals}
Recall the definition of the \textit{lexicographic order} on $S_{n,d}$. Let $u,v \in S_{n,d}$. Write $u=\mathbf{x}^\mathbf{a}=x_1^{a_1}\cdots x_n^{a_n}$ with $\mathbf{a}=(a_1,\dots,a_n) \in \N^n$, and similarly $v=\mathbf{x}^\mathbf{b}$ with $\mathbf{b} \in \N^n$. By definition, we have 
$$
u  \, >_{lex}\, v
$$
if and only if the leftmost nonzero coordinate of $\mathbf{a}-\mathbf{b}$ is positive. Equivalently, let $$u=x_{i_1}\cdots x_{i_d}, \, v=x_{j_1}\cdots x_{j_d} \in S_{n,d}$$ 
with $i_1\le \dots \le i_d$, $j_1\le \dots \le j_d$. Then $u >_{lex} v$ if and only if the leftmost nonzero coordinate of $(i_1-j_1,\dots,i_d-j_d)$ is negative. For simplicity, we shall omit the subscript and write $\ge$ instead of $\ge_{lex}$.

\smallskip
We shall need below the following well-known equivalence.

\begin{lemma}\label{order and multiplication} Let $u,v \in S_{n,d}$. Then for all $1 \le i \le n$, we have $ u > v$ if and only if $x_iu > x_iv$.
\end{lemma}
\begin{proof} Write $u=\mathbf{x}^\mathbf{a}$, $v=\mathbf{x}^\mathbf{b}$ with $\mathbf{a}, \mathbf{b} \in \N^n$. Then $x_iu=\mathbf{x}^{\mathbf{a}+\mathbf{e}_i}$, $x_iv=\mathbf{x}^{\mathbf{b}+\mathbf{e}_i}$ where $\mathbf{e}_i \in \N^n$ is the basis vector with a $1$ at the $i$th coordinate and $0$ elsewhere. The statement follows since 
$$(\mathbf{a}+\mathbf{e}_i)-(\mathbf{b}+\mathbf{e}_i) \, = \, \mathbf{a}-\mathbf{b}. \qedhere$$
\end{proof}

The following notation will be used throughout.
\begin{notation}\label{intervals}
For $u \in S_{n,d}$, we denote by $L(u)$ the \emph{lexsegment} determined by $u$, \emph{i.e.} 
$$L(u) \ = \ \{v \in S_{n,d} \mid v \ge u\}.$$
More generally, for $u_1,u_2 \in S_{n,d}$ such that $u_1 \ge u_2$, we denote by $L(u_1,u_2)$ the \emph{lexinterval} of intermediate monomials, namely
$$L(u_1,u_2) = \{v \in S_{n,d} \mid u_1 \ge v \ge u_2\}.$$
Thus $L(u)=L(x_1^d,u)$ for $u \in S_{n,d}$. Finally, we denote
$$
L^*(u_1,u_2) \ = \ L(u_1,u_2) \setminus \{u_1\}  \ = \ \{v \in S_{n,d} \mid u_1 > v \ge u_2\}.
$$
\end{notation}

\subsection{Gotzmann sets}

\begin{definition} A subset $B \subseteq S_n$ is said to be \emph{Borel-stable} if $u \in B$ implies $x_iu/x_j \in B$ for all $1 \le i \le j \le n$ such that $x_j$ divides $u$.
\end{definition}

\begin{definition} A monomial ideal $I \subseteq R_n$ is said to be \emph{Borel-stable} if its set of monomials $I \cap S_n$ is a Borel-stable set.
\end{definition}

Let $B \subseteq S_{n,d}$. We define and denote the \emph{shade}\footnote{$\shad$ should stand for \textit{shade} as in Combinatorial set theory~\cite {An}, and not for ``shadow'' as written in \cite{B, He}. The \textit{shadow} of $B$ actually corresponds to the set of all monomials $u/x_j$ with $u \in B$ and $x_j$ dividing $u$.} of $B$ by $$\shad(B)= \ \{x_iu \mid u \in B, i=1, \ldots,n \} \subseteq S_{n,d+1}.$$
For $i \ge 2$, the $i$-th shade of $B$ is defined recursively by $\shad^i(B)=\shad(\shad^{i-1}(B)).$
\begin{notation} Let $B \subseteq S_{n,d}.$ We denote by $B^{lex}$ the unique lexsegment in $S_{n,d}$ such that $|B^{lex}|= |B|$.
\end{notation}
Thus, there exists a unique monomial $w_B$ in $S_{n,d}$ such that
$$B^{lex}=L(w_B).$$
\begin{example} Let $B=\left\{x_1^2, x_1x_2,  x_1x_3,  x_2^2 \right\} \subseteq S_{4,2}$.
The lexsegment of length $|B|=4$ in $S_{4,2}$ is 
$$L(x_1x_4)=\left\{x_1^2, x_1x_2,  x_1x_3,  x_1x_4 \right\}.$$ Hence  
$B^{lex}=\left\{x_1^2, x_1x_2,  x_1x_3,  x_1x_4 \right\}=L(x_1x_4)$, and thus $w_B=x_1x_4$. 
\end{example}

\medskip
The following result can be found in~\cite[Theorem 2.7]{He}.
\begin{theorem}\label{thm ineq lex} For any subset $B \subseteq S_{n,d}$, one has 
$$|\shad(B)| \ge |\shad(B^{lex})|.$$
\end{theorem}
\begin{proof} See~\cite{He}.
\end{proof}
\begin{definition} A subset  $B \subseteq S_{n,d}$ is said to be a \emph{Gotzmann set} if equality in Theorem~\ref{thm ineq lex} is achieved, \emph{i.e.} if $$|\shad(B)|=|\shad(B^{lex})|.$$
\end{definition}

\smallskip Recall that a homogeneous ideal $I \subseteq R_n$ is a \textit{Gotzmann ideal} if, from a certain degree on, its Hilbert function attains Macaulay's lower bound. Gotzmann sets are linked to Gotzmann ideals by the following result. See \emph{e.g.} \cite{PS} for more details.

\begin{proposition}\label{gotzmann sets and ideals} Let $B \subseteq S_{n,d}$ with $d \ge 1$. Then the ideal $(B)$ of $R_n$ spanned by $B$ is a Gotzmann ideal if and only if $B$ is a Gotzmann set. 
\end{proposition}

The next lemma is crucial in the characterization of Borel-stable sets which are Gotzmann sets. We first introduce some notation.
\begin{notation} Let $u \in S_n$ be a monomial distinct from $1$. We denote by $\max(u)$ the largest index $i \le n$ such that $x_i$ divides $u$. 
\end{notation}
\begin{notation} Let $B \subseteq S_{n,d}$ be a set of monomials of degree $d \ge 1$. For all $1 \le i \le n$, we denote by $m_i(B)$ the number of monomials $u \in B$ such that $\max(u)=i$.
\end{notation}
Of course, we have $|B|=m_1(B)+\cdots+m_n(B)$.
\begin{lemma}\label{gotzmann sets} Let $B \subseteq S_{n,d}$ be a Borel-stable set. Then $B$ is a Gotzmann set if and only if $$m_i(B)=m_i(B^{lex}) $$ for all $1 \le i \le n$.
\end{lemma}
\begin{proof} 
See \cite{He} and Lemma 1.6 in \cite{B}.
\end{proof} 
Given $B \subseteq S_{n,d}$, it will be useful to collect the numbers $m_i(B)$ for $1 \le i \le n$ as a single monomial. This gives rise to the following definition.
\begin{definition} Let $B \subseteq S_{n,d}$ be a set of monomials of degree $d \ge 1$. Let $m_i=m_i(B)$ for $1 \le i \le n$. The \emph{maxgen monomial} of $B$ is defined by
$$\maxgen(B) \,=\, x_1^{m_1}\cdots x_n^{m_n}.$$
\end{definition}
Note that $\deg(\maxgen(B)) = |B|$. We may now rephrase Lemma~\ref{gotzmann sets} using the maxgen monomial.

\begin{lemma}\label{with maxgen} Let $B \subseteq S_{n,d}$ be a Borel-stable set. Then $B$ is a Gotzmann set if and only if $$\maxgen(B)= \maxgen(B^{lex}).$$
\end{lemma}
\begin{proof} Follows from Lemma~\ref{gotzmann sets} and the definition of $\maxgen(B)$.
\end{proof}

\subsection{The maxgen monomial revisited}

Given $B \subseteq S_{n,d}$ with $d \ge 1$, we now describe $\maxgen(B)$ in a slightly more useful way. First some preliminaries.

\begin{notation} Let $u \in S_n$ be a monomial distinct from 1. We denote by
\begin{itemize}
\item $\min(u)$ the smallest index $i \ge 1$ such that $x_i$ divides $u$;
\item $\lambda(u)=x_j$, where $j=\max(u)$.
\end{itemize}
\end{notation}
Thus $\lambda(u)$ divides $u$, and it is the ``last'', or lexicographically smallest, variable with this property. This yields a function 
$$
\lambda \colon S_{n,d}\setminus \{1\} \longrightarrow S_{n,1} = \{x_1,\dots,x_n\}.
$$
For instance, if $u=x_2^3x_3x_4^2$, then $\min(u)=2$, $\max(u)=4$ and $\lambda(u)=x_4$.

\begin{lemma}\label{maxgen}
Let $B \subseteq S_{n,d}$. Then
$$
\maxgen(B) \,=\, \prod_{w \in B} \lambda(w) \ = \ \prod_{w \in B} x_{\max(w)}.
$$
\end{lemma}
\begin{proof} Directly follows from the definitions.
\end{proof}
Thus, $\maxgen(B)$ may be viewed as the maximal index generating function of all monomials in $B$.

\smallskip
We shall sometimes tacitly use the following easy observation. 
\begin{remark}\label{divides}
If $B \subseteq B' \subseteq S_{n,d}$, then $\maxgen(B)$ divides $\maxgen(B')$.
\end{remark}

\subsection{Gaps and cogaps}

\begin{notation} Let $u \in S_n$. We denote by $B(u)$ the smallest Borel-stable subset of $S_{n}$ containing $u$. 
\end{notation}
Observe that if $u \in S_{n,d}$, then $B(u) \subseteq S_{n,d}$.

\begin{lemma} Let $u \in S_{n,d}$. Then $B(u) \subseteq L(u)$. 
\end{lemma}
\begin{proof} Let $v \in B(u)$. Then $v$ is obtained from $u$ by repeated operations of the form
$$
u' \mapsto x_iu'/x_j
$$
where $u' \in B(u)$, $x_j$ divides $u'$ and $i \le j$. Since $x_iu'/x_j \ge u'$ at each such step, it follows that $v \ge u$, whence $v \in L(u)$.
\end{proof}

For our present purposes, it is of particular interest to consider the set difference $L(u) \setminus B(u)$. The following concept first arose in \cite{B}.
\begin{definition} Let $u \in S_{n,d}$. We set  
$\gaps(u) = L(u)\setminus B(u)$, and we call \emph{gaps of $u$} the elements of this set.
\end{definition}
 
\begin{notation} Let $u \in S_{n,d}$. We denote by $\tilde{u} \in S_{n,d}$ the unique monomial such that 
$$
L(\tilde{u})=B(u)^{lex}.
$$
\end{notation}
Since $B(u)$ and $B(u)^{lex}$ have the same cardinality by definition, we have
$$
|L(\tilde{u})|\,=\,|B(u)|.
$$
Moreover, since $B(u) \subseteq L(u)$ by the above lemma, we have $x_1^d \ge \tilde{u} \ge u$ and so $L(\tilde{u}) \subseteq L(u)$. Here is an illustration of the situation:

\smallskip

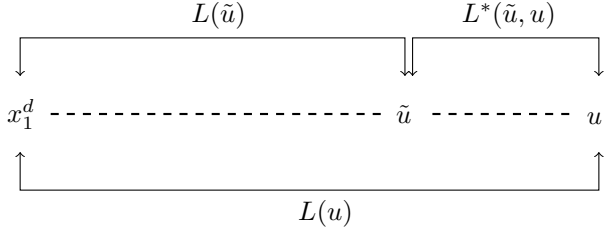
\begin{center}
\begin{tikzpicture} [xscale=2]
\draw [thick, dashed] (-0.1,0) -- (2.1,0);
\node [left] at (-0.14,0) {$x_1^d$};
\node [right] at (2.1,0) {$\tilde{u}$};
\draw [thick, dashed] (2.4,0) -- (3.34,0);
\node [right] at (3.35,-0.05) {$u$};
\draw [<->] (-0.3,-0.5) -- (-0.3,-1) -- (3.49, -1) -- (3.49,-0.5);
\node [below] at (1.7, -1) {$L(u)$};
\draw [<->] (-0.3,0.5) -- (-0.3,1) -- (2.22, 1) -- (2.22,0.5);
\node [above] at (1, 1) {$L(\tilde{u})$};
\draw [<->] (2.27,0.5) -- (2.27,1) -- (3.49,1) -- (3.49,0.5);
\node [above] at (2.9, 1) {$L^*(\tilde{u},u)$};
\end{tikzpicture}
\end{center}

\smallskip
Since $|L(\tilde{u})|\,=\,|B(u)|$, we have 
$$|L(u)\setminus L(\tilde{u})|=|L(u)\setminus B(u)|=|\gaps(u)|.$$
This motivates our definition of $\cogaps(u)$, a lexinterval with the same cardinality as $\gaps(u)$.
\begin{definition}\label{def cogaps} Let $u \in S_{n,d}$. We set $\cogaps(u) \,=\, L(u)\setminus L(\tilde{u})$. That is, $\cogaps(u)$ is the lexinterval of cardinality $g=|\gaps(u)|$ with smallest element $u$.
Equivalently, $$\cogaps(u)=L^*(\tilde{u},u).$$ 

\end{definition}
By construction, we have $|\gaps(u)| = |\cogaps(u)|$ and two partitions of $L(u)$, namely:
$$
L(u) \,=\, \left\{
\begin{array}{rcl}
B(u) & \sqcup & \gaps(u), \\
L(\tilde{u}) & \sqcup & \cogaps(u).
\end{array}
\right.
$$

\begin{example} 
Let $n=4,$  $d=2$ and $u=x_2x_3$. Then
\begin{align*}
B(u) &=\left\{x_1^2,  x_1x_2, x_1x_3, x_2^2,  x_2x_3 \right\}, \\
L(u) &=\left\{x_1^2, x_1x_2,  x_1x_3, x_1x_4, x_2^2,  x_2x_3 \right\},\\
\gaps(u) &= L(u)\setminus B(u)=\left\{x_1x_4\right\}.
\end{align*}
The unique monomial $\tilde{u} \in L(u)$ such that $|L(\tilde{u})|\,=\,|B(u)|$ is $\tilde{u}=x_2^2$. Hence
$\cogaps(u)= L(u)\setminus L(\tilde{u})=\left\{x_2x_3 \right\}$.
\end{example}
A word of caution regarding $L(u)$ and $B(u)$ is needed here. 
\begin{remark}\label{caution} For the lexsegment determined by $u \in S_n$, one should write $L_n(u)$ rather than $L(u)$. Indeed, let $m < n$ be positive integers. Then $S_m \subset S_n$ canonically. Let now $u \in S_m$. Then $L_m(u) \not= L_n(u)$ in general. For instance, with $u=x_2x_3$ as above, we have 
$$L_3(u)=\left\{x_1^2, x_1x_2,  x_1x_3, x_2^2,  x_2x_3 \right\}=L_4(u) \setminus \{x_1x_4\}.$$
Consequently, one should also write $\gaps_n(u)$ rather than $\gaps(u)$. However, we shall systematically omit the index $n$ since it will be fixed in any given discussion below. On the other hand, the set $B(u)$ is independent of $n$.
\end{remark}

\subsection{Gotzmann monomials}

\begin{definition}\label{gotzmann monomial} Let $u \in S_{n,d}$. We say that $u$ is a \emph{Gotzmann monomial} if $B(u)$ is a Gotzmann set.
\end{definition}

\begin{remark}
Note that Gotzmann monomials in $S_n$ may no longer be Gotzmann monomials in $S_{n+1}$. For instance, $x_2x_3$ is Gotzmann in $S_3$ but not in $S_4$.
\end{remark}

Our determination of Gotzmann monomials in $S_3$ and $S_4$ will use the following general characterization.
\begin{theorem}\label{Gotzmann monomials} Let $u \in S_{n}$. Then $u$ is a Gotzmann monomial if and only if 
$$
\maxgen(\gaps(u)) \,=\, \maxgen(\cogaps(u)).
$$
\end{theorem}
\begin{proof} It follows from Definition~\ref{gotzmann monomial} and Lemma~\ref{with maxgen} that $u$ is a Gotzmann monomial if and only if
$$
\maxgen(B(u)) = \maxgen(B(u)^{lex}).
$$
Now by definition of $\tilde{u}$, we have
\begin{equation}
L(\tilde{u})=B(u)^{lex}.
\end{equation} 
Hence $u$ is a Gotzmann monomial if and only if 
$$
\maxgen(B(u)) = \maxgen(L(\tilde{u})).
$$
Since 
$$B(u) \sqcup \gaps(u)=L(\tilde{u}) \sqcup \cogaps(u),$$
as both sides coincide with $L(u)$, it follows that
$$
\maxgen(B(u)) = \maxgen(L(\tilde{u})) \iff \maxgen(\gaps(u)) = \maxgen(\cogaps(u)), 
$$
and the proof is complete.
\end{proof}

\medskip
Thus, from now on, our task will be to develop tools to compute or determine $\gaps(u),\cogaps(u)$ and their respective maxgen monomials, so as to be able to apply the characterization of Gotzmann monomials provided by Theorem~\ref{Gotzmann monomials}.

\section{Some results on gaps}\label{sec gaps}
Let $u \in S_{n,d}$. Recall that $B(u) \subseteq L(u)$ and that $\gaps(u) = L(u) \setminus B(u)$. We first describe the gaps of $u$ in an equivalent way. 
\begin{lemma}\label{lemma gaps} 
Let $u,v \in S_{n,d}$ be monomials of degree $d$ in $S_n$. Set $u = x_{i_1}\cdots x_{i_d}$ with $i_1 \le \dots \le i_d$ and $v = x_{j_1}\cdots x_{j_d}$ with  $j_1 \le \dots \le j_d$. The following are equivalent:
\begin{enumerate}
\item $v$ is a gap of $u$;
\item there exist indices $1 \le s < t \le d$ such that
$$
(j_1,\dots,j_{s-1}) \, = \, (i_1,\dots,i_{s-1}),  \quad j_s  <  i_s,  \quad
j_t  >  i_t.
$$
\end{enumerate}
\end{lemma}
\begin{proof} We have $v \not= u$ since $v \notin B(u)$. The existence of index $s$ with the given property then follows from the hypothesis $v \in L(u)$. The existence of index $t > s$ with its property then directly follows from the hypothesis $v \notin B(u)$.
\end{proof}

\smallskip 

We need yet another notation which will be used to give a structural description of $\gaps(u)$. 
\begin{notation}
For a monomial $u = x_{i_1}\cdots x_{i_d}$ with $i_1 \le \dots \le i_d$, and for any integer $0 \le k \le d$, we denote by $\pre_k(u)$ the \emph{prefix} of $u$ of degree $k$, defined by
$$
\pre_k(u) \ = \  x_{i_1}\cdots x_{i_k}.
$$
\end{notation}

Observe that $\pre_k(u)$ may be characterized as follows: it is the unique monomial $v$ of degree $k$ dividing $u$ and satisfying $\max(v) \le \min(u/v)$.

\begin{definition}\label{A1 A2} For any $v \in S_{n,k}$, we define subsets $A_1(v), A_2(v) \subset S_{n,k}$ as follows:  \vspace{-0.2cm}
\begin{eqnarray*}
A_1(v) & = & B(v)\setminus \{v\}, \\ 
A_2(v) & = & \{w \in S_{n,k} \mid \min(v)+1 \le \min(w) \le n\}.
\end{eqnarray*}
\end{definition}

\begin{proposition}\label{gaps} Let $u \in S_{n,d}$. For all $1 \le k \le d-1$, let $u_k=pre_k(u)$. Then
$$
\gaps(u) \,=\, \bigsqcup_{k=1}^{d-1} \, A_1(u_k)A_2(u/u_k).
$$
\end{proposition}
\begin{proof} First, any monomial $v=v_1v_2$ with $v_1 \in A_1(u_k)$, where $u_k$ is the prefix of $u$ of degree $k$ for some $1 \le k \le d-1$, and $v_2 \in A_2(u/u_k)$, is a gap of $u$  by construction and Lemma~\ref{lemma gaps}.

Conversely, let $v$ be a gap of $u$. Set $u = x_{i_1}\cdots x_{i_d}$ with $i_1 \le \dots \le i_d$ and $v = x_{j_1}\cdots x_{j_d}$ with $j_1 \le \dots \le j_d$. In the notation of Lemma~\ref{lemma gaps}, let $s$ be such that $i_s < j_s$, and let $t >s$ be the \textit{least index} satisfying $j_t > i_t$.  Set $k=t-1$. Then by construction, the factor $x_{j_1}\cdots x_{j_k}$ of degree $k$ of $v$ belongs to $B(u_k)$, since $i_\alpha \le j_\alpha$ for all $\alpha < t$ by minimality of $t$, and in fact belongs to $A_1(u_k)$ since $i_s <j_s$, whereas the cofactor $x_{j_{k+1}}\cdots x_{j_d}$ belongs to $A_2(u/u_k)$ since 
$j_d \ge \dots \ge j_{k+1}=j_t > i_t. \qedhere$
\end{proof}
\begin{corollary} Let $u = x_{i_1}\cdots x_{i_d}$ with $i_1 \le \dots \le i_d$. Then
$$
|\gaps(u)|  \,=\, \sum_{k=1}^{d-1} \, \big(\big|B(x_{i_1}\cdots x_{i_{k}})\big|-1\big)\cdot \big|S_{n-i_{k+1},d-k}\big|.
$$
\end{corollary}
\begin{proof} The number of monomials $w_2 \in S_n$ of degree $d-k$  in the variables $x_{i_{k+1}+1}, \ldots,x_n$ is equal to the number of monomials of the same degree in $S_{n-i_{k+1}}$, \emph{i.e.} to $\big|S_{n-i_{k+1},d-k}\big|$. 
\end{proof}

\smallskip

This prompts us to find good formulas for $|B(w)|$ for any monomial $w$. Here is an inductive approach.

\begin{proposition}\label{prop card <w>} Let $w \in S_n$ and $m = \max(w)$. Let $r \ge 1$ be the largest exponent such that $x_m^r$ divides $w$. Let $v = w/x_m^r$, so that $\max(v) < m$.
Then
$$
B(w) \ = \ \bigsqcup_{i=0}^r B(vx_{m-1}^{r-i})x_m^i. 
$$
\end{proposition}
\begin{proof} Indeed, let $w' \in B(w)$. Then $\max{w'} \le m$. If $\max{w'} < m$, then clearly $w' \in B(vx_{m-1}^r)$. Otherwise, if $\max{w'} = m$, let $i$ be the largest exponent such that $x_m^i$ divides $w'$, so that $1 \le i \le r$. Let $v'=w'/x_m^i$. Then clearly $v' \in B(vx_{m-1}^{r-i})$, so that $w' \in B(vx_{m-1}^{r-i})x_m^i$.
\end{proof}

\begin{corollary}\label{cor card <w>} As above, let $w=vx_m^r \in S_n$ with $\max(w)=m$ and $\max(v) <m$. Then
$$
|B(w)| \ = \ \sum_{i=0}^r |B(vx_{m-1}^{r-i})|.
$$
\end{corollary}
\begin{proof} Directly follows from the above partition of $B(w)$.
\end{proof}
This corollary reduces the computation of $|B(w)|$ for monomials in $S_n$ to that for monomials in $S_{n-1}$.

\section{Predecessors and successors}\label{sec pred and succ}
\begin{definition} Let $u,v \in S_{n,d}$ such that $u > v$. We say that $u$ \emph{covers} $v$ if there are no intermediate monomials between them, \emph{i.e.} if for any $w \in S_{n,d}$ such that $u \ge w \ge v$, we have $w=u$ or $w=v$. In that case, we say that $u$ is the \emph{predecessor} of $v$, that $v$ is the \emph{successor} of $u$, and we write 
$$u = \pred(v), \quad v = \suc(u).$$
\end{definition}
Since $x_1^d$ and $x_n^d$ are the largest and smallest elements in $S_{n,d}$, respectively, the predecessor of $x_1^d$ and the successor of $x_n^d$ are undefined.

Note that, for all $u \in S_{n,d}$ with $u \notin \{x_1^d, x_n^d\}$, we have
$$
{\rm succ}(\pred(u)) \ = \ \pred({\rm succ}(u)) \ = \ u.
$$

\begin{proposition}\label{tilde u} Let $u \in S_n$ and $g=|\gaps(u)|$. Then $\tilde{u}$ is the $g $\emph{th} predecessor of $u$, \emph{i.e.}
\begin{equation}\label{pred}
\tilde{u} \,=\, \pred^g(u).
\end{equation}
\end{proposition}
\begin{proof} Indeed, recall the partition $L(u) = L(\tilde{u}) \sqcup L^*(\tilde{u},u)$. Thus $\tilde{u}$, the smallest element of $L(\tilde{u})$, is the predecessor of the largest element of $L^*(\tilde{u},u)=\cogaps(u)$. Since $\cogaps(u)$ has cardinality $g$ and is a lexinterval ending at $u$, its largest element is $\pred^{g-1}(u)$. Hence $\tilde{u}=\pred(\pred^{g-1}(u))$, as desired.
\end{proof}

Lexintervals ending at a monomial $u$ are made up of iterated predecessors of $u$. This motivates the following notation.
\begin{notation} Let $u \in S_{n,d}$ and $r \le |L(u)|$. We denote 
$$\pred_r(u) = \{\pred^i(u) \mid 0 \le i < r\}.$$ 
\end{notation}
That is, $\pred_r(u)$ is the set of $r$ predecessors of $u$, including $u$ itself. This set is well defined since $r \le |L(u)|$ by hypothesis. Of course $\pred_r(u)$ is a lexinterval, since
\begin{equation}\label{cogap and pred}
\pred_r(u) = \{w \in S_{n,d} \mid \pred^r(u) > w \ge u \} = L^*(\pred^r(u),u).
\end{equation}

We may now reinterpret the lexinterval $\cogaps(u)$ in terms of the above concept.
\begin{proposition}\label{pred and cogaps} Let $u \in S_{n,d}$ and $g = |\gaps(u)|$. Then
$$
\cogaps(u) = \pred_g(u). 
$$
\end{proposition}
\begin{proof} By Definition~\ref{def cogaps}, we have $|\cogaps(u)|=|\gaps(u)|=g$, and $\cogaps(u)= L^*(\tilde{u},u)$ where $\tilde{u} = \pred^g(u)$. The stated equality follows from \eqref{cogap and pred}. 
\end{proof}

As we shall need to determine $\pred^i(u)$ for any given $u \in S_{n,d}$, we need an explicit description of $\pred(u)$. We start with the description of the successor  of a monomial in $S_{n,d}$ distinct from $x_n^d$.

\begin{proposition} Let $u \in S_{n,d}$  distinct from $x_n^d$. Set $u=v x_n^{a_n}$ with $a_n \in \N$ and $\max(v) = m$ with $m \le n-1$. Then
$$
{\rm succ}(u) \,=\, (v/x_m)x_{m+1}^{a_n+1}.
$$  
\end{proposition}
\begin{proof} This easily follows from the definition of the lexicographic order on $S_{n,d}$.
\end{proof}

\begin{proposition}\label{predecessor} Let $u \in S_{n,d}$ such that $u \not= x_1^d$, and let $m=\max(u)$. Write $u=vx_m^a$ with $a \ge 1$ and $v \in S_{n,d-a}$, so that $\max(v) \le m-1$. Then
$$
\pred(u) \ = \ vx_{m-1}x_n^{a-1}.
$$
\end{proposition}
\begin{proof} Follows from the description above of the successor of a monomial in $S_{n,d}$ distinct from $x_n^d$.
\end{proof}

\smallskip

The next corollary compares $\max(\pred(u))$ with $\max(u)$. There are only two possible outcomes, linked to whether $\lambda(u)^2$ divides $u$ or not; recall that $\lambda(u)$ always divides $u$ by construction. 

\begin{corollary} For all $u \in S_{n,d}$ such that $u \not=x_1^d$, we have:
$$
\max(\pred(u)) \ = \
\left\{
\begin{array}{ll}
n & \textrm{if } \, \lambda(u)^2 \ | \ u,\\
\max(u)-1 & \textrm{if not}.
\end{array}
\right.
$$
\end{corollary}

\begin{example} The lexicographically smallest monomial $u \in S_{4,d}$ such that
$$
\max(\pred(u)) \ \le 2
$$
is $x_2^{d-1}x_3$. That is, for all $v \in S_{4,d}$ such that $v \le x_2^{d-1}x_3$, we have $\max(v) \ge 3$, \emph{i.e.} $\lambda(v) \in \{x_3,x_4\}$. 
\end{example}

\section{The maxgen monomial of lexintervals}\label{sec maxgen}

\subsection{The function $\mu_n$}

We now introduce a function of two monomials $u_1 \ge u_2$ in $S_{n,d}$ which will later be used to give an equivalent description of cogaps. Recall the notation
$$
L^*(u_1,u_2) \ = \ L(u_1,u_2) \setminus \{u_1\}  \ = \ \{v \in S_{n,d} \mid u_1 > v \ge u_2\}.
$$
\begin{definition}\label{def of mu} For $u_1,u_2 \in S_{n,d}$ such that $u_1 \ge u_2$, we define $\mu_n(u_2,u_1) \in S_n$ to be the maxgen monomial of the lexinterval $L^*(u_1,u_2)$, \emph{i.e.}
$$
\mu_n(u_2,u_1) \ = \ \maxgen(L^*(u_1,u_2)).
$$
\end{definition} 
Equivalently, recalling that maxgen collects the last variables of a set of monomials and takes their product:
$$
\mu_n(u_2,u_1) \ = \ \prod_{u_1 > v \ge u_2} \lambda(v).
$$
Note that by construction, the last variable of $u_1$ is \textit{not} taken into account in $\mu_n(u_2,u_1)$.

\begin{remark} As in Remark~\ref{caution}, we have $S_{n,d} \subset S_{n+1,d}$ canonically. Now if $u_1,u_2 \in S_{n,d}$, then $\mu_n(u_2,u_1)$ and $\mu_{n+1}(u_2,u_1)$ differ in general. However, when the number $n$ of variables involved is clear from the context, we shall simply write $\mu(u_2,u_1)$ for $\mu_n(u_2,u_1)$. 
\end{remark}

The function $\mu$ on $S_{n,d}$ has the following transitive property.
\begin{lemma}\label{composition of mu} For all $u_1, u_2, u_3 \in S_{n,d}$ such that $u_1 \ge u_2 \ge u_3$, we have
$$
\mu(u_3,u_1) \ = \ \mu(u_3,u_2)\mu(u_2,u_1).
$$
\end{lemma}
\begin{proof} Directly follows from the definition.
\end{proof}

\begin{notation} For monomials $u_1 \ge u_2$ in $S_{n,d}$, we shall occasionally denote  the equality $\mu(u_2,u_1)=w$ as follows:
$$
u_2 \overset{w}{\xrightarrow{\hspace*{1.5cm}}} u_1.
$$
\end{notation}
Lemma~\ref{composition of mu} then amounts to \emph{arrow composition}: if $u_1 \ge u_2 \ge u_3$ in $S_{n,d}$, then
$$
u_3 \overset{w_2}{\xrightarrow{\hspace*{1.5cm}}} u_2 \overset{w_1}{\xrightarrow{\hspace*{1.5cm}}} u_1
$$
is equivalent to
$$
u_3 \overset{w_2w_1}{\xrightarrow{\hspace*{1.6cm}}} u_1.
$$
For instance, starting from $x_3^2$ and taking successive predecessors in $S_3$, one has
$$
x_3^2 \overset{x_3}{\xrightarrow{\hspace*{1.5cm}}} x_2x_3 \overset{x_3}{\xrightarrow{\hspace*{1.5cm}}} x_2^2 \overset{x_2}{\xrightarrow{\hspace*{1.5cm}}} x_1x_3 \overset{x_3}{\xrightarrow{\hspace*{1.5cm}}} x_1x_2.
$$
By arrow composition, this may be summarized as
$$
x_3^2 \overset{x_2x_3^3}{\xrightarrow{\hspace*{1.7cm}}} x_1x_2,
$$
expressing the equality $\mu_3(x_3^2,x_1x_2)=x_2x_3^2$.

\begin{remark}
If $r \le |L(u)|$, then $\pred_r(u)$ is defined and we have
$$
u \arrow{\maxgen(\pred_r(u))}{3.3cm} \pred^r(u)
$$
by construction.
\end{remark}

In particular, with $r=|\gaps(u)|$, this yields the following tool in view of effectively applying Theorem~\ref{Gotzmann monomials}.
\begin{proposition} Let $u \in S_{n,d}$. Then
$$
\maxgen(\cogaps(u))  \ = \  \mu(u,{\tilde u})
$$
\emph{i.e.}, in arrow notation, \, \, $u \arrow{\maxgen(\cogaps(u))}{3.3cm} {\tilde u}$.
\end{proposition}
\begin{proof} Follows from the above remark and the facts that, if $g = |\gaps(u)|$, then $\tilde{u}=\pred^g(u)$ and $\cogaps(u)=\pred_g(u)$ by Propositions~\ref{tilde u} and \ref{pred and cogaps}, respectively.
\end{proof}

\begin{lemma}\label{pred(vw)} Let $v,w \in S_n$. If $\max(v) \le \max(w)$, then
$$
\pred(vw) \ = \ v \pred(w).
$$
\end{lemma}
\begin{proof} Let $j = \max(w)$. Then $w = w'x_j^{a}$ with $\max(w') < j$ and $a \ge 1$. Since $\max(v) \le j$ by hypothesis, we may write $v = v' x_j^b$ with $\max(v') < j$ and $b \ge 0$.
We have
$$
\pred(w) = w' x_{j-1}x_n^{a-1}.
$$
Now, since $vw=v'w'x_j^{a+b}$ and since $\max(v'w') <j $, we have
$$
\pred(vw) = v'w' x_{j-1}x_n^{a-1+b} = v \pred(w). \qedhere
$$
\end{proof}

\begin{lemma}\label{property 2 of mu} For all $u_1, u_2 \in S_{n,d}$ such that $u_1 > u_2$ and all $v \in S_n$ such that $\max(v) \le \min \mu(u_2,u_1)$, we have
$$
\mu(vu_2,vu_1) \ = \ \mu(u_2,u_1).
$$
\end{lemma}
\begin{proof} Let $w \in S_{n,d}$ satisfy $u_1 > w \ge u_2$. We have $vu_1 > vw \ge vu_2$, since the product is compatible with the lex order. It follows from the hypothesis that $\max(v) \le \max(w)$. Lemma~\ref{pred(vw)} implies $\pred(vw) \ = \ v \pred(w)$. Therefore $L^*(vu_1,vu_2)=vL^*(u_1,u_2)$, whence the conclusion.
\end{proof}

\smallskip

In order to apply this lemma, we need some control on $\min \mu(u_2,u_1)$. This is provided by the next proposition. First a lemma.

\begin{lemma}\label{lex order} Let $u,v \in S_{n,d}$. If $u \ge v$ then $\min u \le \min v$.
\end{lemma}
\begin{proof} By definition of the lex order, small indices weigh more. Hence if $\min u > \min v$ then $u < v$.
\end{proof}

\begin{proposition}\label{min mu} Let $u_1, u_2 \in S_{n,d}$. If $u_1 > u_2$ then $\min \mu(u_2,u_1) > \min u_1$. 
\end{proposition}
\begin{proof}  By Definition~\ref{def of mu}, we have 
$$
\mu (u_2,u_1) \ = \ \prod_{u_1 > v \ge u_2} \lambda(v).
$$
For $v < u_1$, the above lemma implies $\min v \ge \min u_1$. Hence $\max v > \min u_1$, for otherwise we would have $\max v = \min v = \min u_1$, implying $v=x_i^d$ for some $i$. But from $v=x_i^d < u_1$, it follows that $\min u_1 < i$, contradicting $\min v = \min u_1$. Having established $\max v > \min u_1$ for all $v < u_1$, it follows that $\min \mu(u_2,u_1) > \min u_1$, as stated.
\end{proof}

\smallskip
Here are straightforward applications.
\begin{corollary}\label{removing v} For all $u_1, u_2 \in S_{n,d}$ such that $u_1 > u_2$ and all $v \in S_n$ such that $\max v \le (\min u_1)+1$, we have
$$
\mu(vu_2,vu_1) \ = \ \mu(u_2,u_1).
$$
\end{corollary}
\begin{proof} By the above proposition, we have $\min \mu(u_2,u_1) \ge (\min u_1)+1$. Hence $\max v \le \min \mu(u_2,u_1)$ by hypothesis, and the claimed equality then follows from Lemma~\ref{property 2 of mu}.
\end{proof}
\begin{corollary}\label{x_m^d} Let $m \le n-1$ and let $u \in S_{n,d}$ such that $u<x_m^d$. Then $$\min \mu (u,x_m^d)\ge m+1.$$
\end{corollary}
\begin{proof} Directly follows from the above proposition.
\end{proof}

\smallskip
We now compute $\mu(vx_m^k,vx_{m-1}^k)$ provided $\max v \le m-1$. For instance, in $S_{n,d}$ we have by the theorem below:
\begin{eqnarray*}
vx_n^k & \arrow{x_n^k}{1.6cm} & vx_{n-1}^k  \quad \textrm{ if } \max v \le n-1, \\
vx_{n-1}^k & \arrow{x_{n-1}^kx_n^{\binom{k}{2}}}{1.9cm} & vx_{n-2}^k \quad \textrm{ if } \max v \le n-2,\\
vx_{n-2}^k & \arrow{x_{n-2}^k x_{n-1}^{\binom{k}{2}}x_{n}^{\binom{k+1}{3}}}{2.5cm} & vx_{n-3}^k  \quad \textrm{ if } \max v \le n-3.\\
\end{eqnarray*}

In view of a general statement, the following intermediate formula will be useful.
\begin{proposition}\label{useful} For all $2 \le m \le n$ and all $k \ge 1$, we have
\begin{eqnarray}\label{new induction}
\mu(x_{m}^k, x_{m-1}^{k}) & = & x_m^k \prod_{1 \le i \le k-1} \mu(x_n^{i}, x_m^{i}).
\end{eqnarray}
\end{proposition}
\begin{proof}
By induction on $k$. For $k=1$, it is clear that $\mu(x_m,x_{m-1})=x_m$, since the predecessor of $x_m$ is $x_{m-1}$. Assume now $k \ge 2$ and that formula \eqref{new induction} holds for $k-1$, \emph{i.e.} 
$$
\mu(x_{m}^{k-1}, x_{m-1}^{k-1}) \, = \, x_m^{k-1} \prod_{1 \le i \le k-2} \mu(x_n^{i}, x_m^{i}).
$$
Thus, in order to establish~\eqref{new induction}, we only need to show
\begin{equation*}
\mu(x_{m}^k, x_{m-1}^{k})  =  x_m \mu(x_n^{k-1}, x_m^{k-1})\mu(x_{m}^{k-1}, x_{m-1}^{k-1}),
\end{equation*}
that is, by Lemma~\ref{composition of mu}, 
\begin{equation} \label{step}
\mu(x_{m}^k, x_{m-1}^{k})  =  x_m \mu(x_n^{k-1}, x_{m-1}^{k-1}). 
\end{equation}
In $S_n$, the predecessor of $x_{m}^k$ is $x_{m-1}x_n^{k-1}$. Hence we have
$$
x_{m}^k \arrow{x_m}{1cm} x_{m-1}x_n^{k-1} \arrow{\mu(x_{m-1}x_n^{k-1}, x_{m-1}x_m^{k-1})}{3.5cm} x_{m-1}x_m^{k-1}.
$$
This means
\begin{equation}\label{intermediate}
\mu(x_{m}^k, x_{m-1}x_m^{k-1}) = x_m\mu(x_{m-1}x_n^{k-1}, x_{m-1}x_m^{k-1}).
\end{equation}
But it follows from Corollary~\ref{removing v} that
\begin{equation}\label{remove x (m-1)}
\mu(x_{m-1}x_n^{k-1}, x_{m-1}x_m^{k-1}) \, =\, \mu(x_n^{k-1}, x_m^{k-1}).
\end{equation}
By Lemma~\ref{composition of mu}, we have
$$
\mu(x_{m}^k, x_{m-1}^{k}) = \mu(x_{m}^k, x_{m-1}x_m^{k-1})\mu(x_{m-1}x_m^{k-1}, x_{m-1}^{k}).
$$
Moreover, by \eqref{intermediate} and \eqref{remove x (m-1)} again, we have
\begin{equation}\label{step2}
\mu(x_{m}^k, x_{m-1}x_m^{k-1}) = x_m\mu(x_n^{k-1}, x_m^{k-1}).
\end{equation}
Hence
$$
\mu(x_{m}^k, x_{m-1}^{k}) = x_m\mu(x_n^{k-1}, x_m^{k-1})\mu(x_{m-1}x_m^{k-1}, x_{m-1}^{k}).
$$
By Corollary~\ref{removing v}, we have
$$
\mu(x_{m-1}x_m^{k-1}, x_{m-1}^{k}) = \mu(x_m^{k-1}, x_{m-1}^{k-1}).
$$
This proves \eqref{step} and hence the claimed formula \eqref{new induction}.
\end{proof}

Here is the promised general statement. As usual, by convention, an empty product equals $1$, as occurs below for $m=n$.
\begin{theorem}\label{from m to m-1} For all $2 \le m \le n$, for all $k \ge 1$, and for all $v \in S_n$ such that $\max v \le m-1$, we have
\begin{equation}\label{formula from m to m-1}
\mu(vx_{m}^k,vx_{m-1}^k) = x_m^k \prod_{i=1}^{n-m} x_{m+i}^{\binom{k-1+i}{i+1}}.
\end{equation}
\end{theorem}
\begin{proof} By Corollary~\ref{removing v}, we have 
$$\mu(vx_m^k,vx_{m-1}^k)=\mu(x_m^k,x_{m-1}^k).$$ 
Hence, it suffices to establish \eqref{formula from m to m-1} for $v=1$. We proceed by induction on $k$ and descending induction on $m$. For $k=1$ and any $m \ge 1$, we have $\mu(x_m,x_{m-1})=x_m$, and this plainly coincides with the right-hand side of \eqref{formula from m to m-1} since $\binom{1-1+i}{i+1}=0$ for all $i \ge 1$. Similarly, for $m=n$ and any $k \ge 1$, we have
$$
x_{n}^k \arrow{x_n^k}{1.8cm}  x_{n-1}^k,
$$
since $L^*(x_{n-1}^k,x_{n}^k) = \{x_{n-1}^{k-i}x_n^i \mid 1 \le i \le k\}$ and since 
$$\mu(x_n^k,x_{n-1}^k)=\maxgen(L^*(x_{n-1}^k,x_{n}^k))=\prod_{x_{n-1}^k > v \ge x_{n}^k} \lambda(v).$$ Equivalently, in formula:
$$
\mu(x_{n}^k, x_{n-1}^k) \ = \ x_n^k.
$$

Assume now that~\eqref{formula from m to m-1} holds for some $m$ such that $n \ge m \ge 2$ and some $k \ge 2$. We now show that~\eqref{formula from m to m-1} also holds for $m-1$. By arrow composition, we have 
$$\mu(x_{n}^i, x_{m-1}^i)=\mu(x_{n}^i, x_{m}^i)\mu(x_{m}^i, x_{m-1}^i).$$
Therefore
\begin{eqnarray*}
\prod_{i=1}^{k-1} \mu(x_n^{i}, x_{m-1}^{i}) & = & \prod_{i=1}^{k-1} \mu(x_n^{i}, x_{m}^{i})\prod_{i=1}^{k-1} \mu(x_m^{i}, x_{m-1}^{i}) \\
& = & \prod_{i=1}^{n-m} x_{m+i}^{\binom{k-1+i}{i+1}} \prod_{i=1}^{k-1} \mu(x_m^{i}, x_{m-1}^{i}) \quad \mbox{(by induction on $m$)}\\
& = & \prod_{i=1}^{n-m} x_{m+i}^{\binom{k-1+i}{i+1}} \prod_{i=1}^{k-1} x_m^i \prod_{j=1}^{n-m} x_{m+j}^{\binom{i-1+j}{j+1}} \quad \mbox{(by induction on $k$)} \\
& = & x_m^{\binom{k}{2}} \prod_{i=1}^{n-m} x_{m+i}^{\binom{k-1+i}{i+1}} \prod_{i=1}^{k-1}  \prod_{j=1}^{n-m} x_{m+j}^{\binom{i-1+j}{j+1}} \\
& = & x_m^{\binom{k}{2}} \prod_{i=1}^{n-m} x_{m+i}^{\binom{k-1+i}{i+1}} \prod_{j=1}^{n-m} \prod_{i=1}^{k-1} x_{m+j}^{\binom{i-1+j}{j+1}} \\
& = & x_m^{\binom{k}{2}}\prod_{i=1}^{n-m} x_{m+i}^{\binom{k-1+i}{i+1}} \prod_{j=1}^{n-m} x_{m+j}^{\sum_{i=1}^{k-1}\binom{i-1+j}{j+1}}.
\end{eqnarray*}
Exchanging the names of the indices $i,j$ in the last product, we get
\begin{eqnarray*}
\prod_{i=1}^{k-1} \mu(x_n^{i}, x_{m-1}^{i}) & = & x_m^{\binom{k}{2}}\prod_{i=1}^{n-m} x_{m+i}^{\binom{k-1+i}{i+1}} \prod_{i=1}^{n-m} x_{m+i}^{\sum_{j=1}^{k-1}\binom{i-1+j}{i+1}} \\
& = & x_m^{\binom{k}{2}}\prod_{i=1}^{n-m} x_{m+i}^{\binom{k-1+i}{i+1}+\sum_{j=1}^{k-1}\binom{i-1+j}{i+1}} \\
& = & x_m^{\binom{k}{2}}\prod_{i=1}^{n-m} x_{m+i}^{\binom{k-1+i}{i+1}+\binom{k-1+i}{i+2}} \\
& = & x_m^{\binom{k}{2}}\prod_{i=1}^{n-m} x_{m+i}^{\binom{k+i}{i+2}} \\
& = & \prod_{i=0}^{n-m} x_{m+i}^{\binom{k+i}{i+2}}.
\end{eqnarray*}
Finally, substituting $i$ by $i-1$ in the last product yields 
$$
\prod_{i=1}^{k-1} \mu(x_n^{i}, x_{m-1}^{i})  = \prod_{i=1}^{n-(m-1)} x_{(m-1)+i}^{\binom{k-1+i}{i+1}}.
$$
Hence ~\eqref{formula from m to m-1} also holds for $m-1$, as claimed. This concludes the proof of the theorem.
\end{proof}

\subsection{Some maxgen computations}
The following result uses the sets $A_1(v), A_2(v)$ introduced in Definition~\ref{A1 A2}. It will be needed in view of applying Theorem~\ref{Gotzmann monomials}.
\begin{proposition}\label{prop maxgen gaps}  Let $u = x_{i_1}\cdots x_{i_d}$ with $i_1 \le \dots \le i_d$. For all $1 \le k \le d-1$, let $u_k=pre_k(u)$. Then
$$
\maxgen(\gaps(u)) \,=\, \prod_{k=1}^{d-1} (\maxgen(A_2(u/u_k)))^{|A_1(u_k)|}.
$$
\end{proposition}
\begin{proof} Consider the description of gaps given in Proposition~\ref{gaps}. For any monomial $w=w_1w_2 \in A_1(u_k)A_2(u/u_k)$ with $w_1 \in A_1(u_k)$ and $w_2 \in A_2(u/u_k)$, we have $\max(w)=\max(w_2)$, since $\max(w_1) < \min(w_2)$ by construction. Therefore, for all $k$ we have
$$
\maxgen(A_1(u_k)A_2(u/u_k)) \ = \ \maxgen(A_2(u/u_k))^{|A_1(u_k)|}. \qedhere
$$
\end{proof}

\smallskip

Since $A_2(u/u_k)$ is the set of all monomials of degree $\deg(u/u_k)$ in the variables $x_i$ with $\min(u/u_k)+1 \le i \le n$, we will be able to determine $\maxgen(A_2(u/u_k))$ if we can determine $\maxgen(S_{n,d})$ for any $n,d$. Let us proceed to do just that. We start with a recurrence formula.
\begin{proposition} For all integers $n,d \ge 1$, we have
$$
\maxgen(S_{n,d}) \ = \ \maxgen(S_{n-1,d})x_n^{\binom{n+d-2}{d-1}}.
$$
\end{proposition}
\begin{proof} Obviously, we have
$$
S_{n,d} \ = \ \bigsqcup_{i=0}^d S_{n-1,d-i}\cdot x_n^i.
$$
This follows from writing any $u \in S_{n,d}$ as $u=vx_n^i$ with $v \in S_{n-1,d-i}$. Hence
\begin{eqnarray*}
\maxgen(S_{n,d}) & = & \prod_{i=0}^d \maxgen(S_{n-1,d-i}\cdot x_n^i) \\
& = & \maxgen(S_{n-1,d}) \prod_{i=1}^d x_n^{|S_{n-1,d-i}|} \\
& = & \maxgen(S_{n-1,d}) x_n^{\sum_{i=1}^d |S_{n-1,d-i}|} \\
& = & \maxgen(S_{n-1,d}) x_n^{\sum_{i=1}^d \binom{n-2+d-i}{d-i}} \\
& = & \maxgen(S_{n-1,d}) x_n^{\sum_{j=0}^{d-1} \binom{n-2+j}{j}}.
\end{eqnarray*}
We conclude the proof by applying the well known formula
$$
\sum_{j=0}^{d-1} \binom{n-2+j}{j} \ = \ \binom{n+d-2}{d-1}. \qedhere
$$
\end{proof}

\begin{corollary}\label{cor maxgen S_n,d} For all $n,d$, we have
$$
\maxgen(S_{n,d}) \ = \ \prod_{i=1}^n x_i^{\binom{d-2+i}{d-1}}.
$$
\end{corollary}
\begin{proof} Use above induction formula.
\end{proof}

\begin{corollary}\label{cor maxgen S_l,n,d} For all $1 \le l \le n$ and all $d$, let $S_{l,n,d}$ be the set of all monomials of degree $d$ in the variables $x_l,\dots,x_n$. Then
we have
$$
\maxgen(S_{l,n,d}) \ = \ \prod_{j=l}^n x_j^{\binom{d-1+j-l}{d-1}}.
$$
\end{corollary}
\begin{proof} Directly follows from the preceding corollary by properly translating indices.
\end{proof}

We may now inject this information into Proposition~\ref{prop maxgen gaps}. This yields the following result.
\begin{theorem}\label{thm maxgen gaps u} Let $u = x_{i_1}\cdots x_{i_d}$ with $i_1 \le \dots \le i_d$. Then
$$
\maxgen(\gaps(u)) \ = \ \prod_{k=1}^{d-1}\left( \prod_{j=i_{k+1}+1}^n x_j^{\binom{d-k-2+j-i_{k+1}}{d-k-1}}\right)^{|B(x_{i_1}\cdots x_{i_k})|-1},
$$
where the internal product is set to 1 if $i_{k+1}=n$.
\end{theorem}
\begin{proof} The proof follows from Proposition~\ref{prop maxgen gaps} together with the above corollary. Using Definition~\ref{A1 A2} for $A_1(v), A_2(v)$, and since $u_k=x_{i_1}\dots x_{i_k}$, we have
$$
|A_1(u_k)| \ = \ |B(x_{i_1}\dots x_{i_k})|-1.
$$
Moreover, $A_2(u/u_k)$ is the set of all monomials of degree $\deg(u/u_k)$ in the variables $x_i$ with $\min(u/u_k)+1 \le i \le n$. Therefore, in order to determine $\maxgen(A_2(u/u_k))$, it remains to apply Corollary~\ref{cor maxgen S_l,n,d}, using $l=i_{k+1}+1$ since $u/u_k=x_{i_{k+1}}\dots x_{i_d}$ and so $\min(u/u_k) = i_{k+1}$.
\end{proof}

\section{On the first and last variables}\label{sec x1 xn}

For the determination of Gotzmann monomials in $S_n$, both variables $x_1$ and $x_n$ behave in some specific ways. This section describes how.

\subsection{Neutrality of $x_1$}
Our purpose here is to show that \textit{a monomial $u \in S_{n}$ is Gotzmann if and only if $x_1u$ is}. We start with some intermediate results.
\begin{lemma}\label{div by x_1} Let $u,v \in S_{n,d}$ such that $u \ge v$. If $x_1$ divides $v$ then $x_1$ divides $u$.
\end{lemma}
\begin{proof} Write $u=x_{i_1}\cdots x_{i_d}$, $v=x_{j_1}\cdots x_{j_d}$ with $1 \le i_1\le \dots \le i_d \le n$ and $1 \le j_1\le \dots \le j_d \le n$. Without loss of generality, we may assume $u > v$. Hence there exists an index $1 \le k \le d$ such that 
$$
i_1=j_1, \dots, i_{k-1}=j_{k-1}, i_k < j_k. 
$$
Therefore $i_1 \le j_1$. Assume $x_1$ divides $v$. This is equivalent to $j_1=1$. Since $1 \le i_1 \le j_1$, then $i_1=1$ whence $x_1$ divides $u$ and we are done. 
\end{proof}
\begin{lemma}\label{L(x_1u)} Let $u \in S_{n,d}$. Then $x_1L(u) \,=\,  L(x_1u)$.
\end{lemma}
\begin{proof} Let $v \in L(u)$. Then $v \ge u$, whence $x_1v \ge x_1u$ by Lemma~\ref{order and multiplication}, \emph{i.e.} $x_1v \in L(x_1u)$. Therefore $x_1L(u) \subseteq L(x_1u)$. Conversely, let $v' \in L(x_1u) \subseteq S_{n,d+1}$. Then $v' \ge x_1u$. Hence, mutatis mutandis, $x_1$ divides $v'$ by Lemma~\ref{div by x_1}. Thus there exists $v \in S_{n,d}$ such that $v'=x_1v$. Now $x_1v \ge x_1u$ by hypothesis, whence $v \ge u$ by Lemma~\ref{order and multiplication} again, \emph{i.e.} $v \in L(u)$ and so $v' \in x_1L(u)$. Therefore $L(x_1u) \, \subseteq \, x_1L(u)$. 
\end{proof}
In particular, the lemma implies that \textit{multiplying any lexsegment by $x_1$ again yields a lexsegment}.
\begin{lemma}\label{x_1B^lex}
Let $B \subseteq S_{n,d}$. Then $(x_1B)^{lex}=x_1B^{lex}$.
\end{lemma}
\begin{proof} We have $|B^{lex}|=|B|=|x_1B|$. Applying this to the set $x_1B$ yields 
\begin{equation}\label{five =}
|(x_1B)^{lex}|=|x_1B|=|B|=|B^{lex}|=|x_1B^{lex}|.
\end{equation}
Now $B^{lex}$ is a lexsegment, whence $x_1B^{lex}$ also is by Lemma~\ref{L(x_1u)} and the comment following it. Moreover, $x_1B^{lex}$ has the same cardinality as the lexsegment $(x_1B)^{lex}$ by \eqref{five =}. Whence these two lexsegments coincide.
\end{proof}

\begin{proposition}\label{gaps x_1u} Let $u \in S_{n,d}$. Then $\gaps(x_1u)=x_1\gaps(u)$.
\end{proposition}
\begin{proof}
Let $v' \in \gaps(x_1u)$. Then $v' > x_1u$, whence $x_1$ divides $v'$ by Lemma~\ref{div by x_1}. Let $v \in S_{n,d}$ such that $v'=x_1v$. Since $x_1v \in \gaps(x_1u)$, it follows that $v \in \gaps(u)$, since $v \ge u$ and $v$ cannot belong to $B(u)$ for otherwise $x_1v$ would belong to $B(x_1u)$. Hence $v' \in x_1\gaps(u)$. 

Conversely, let $v \in \gaps(u)$. Then $v > u$, whence $x_1v > x_1u$ and so $x_1v \in L(x_1u)$. Since $v \notin B(u)$, it follows that $x_1v \notin B(x_1u)$. Whence $x_1v \in \gaps(x_1u)$.
\end{proof}

\begin{theorem}\label{u and x_1u} Let $u \in S_n$. Then $u$ is Gotzmann if and only if $x_1u$ is Gotzmann.
\end{theorem}
\begin{proof}  First some preliminary steps. 

\medskip
\noindent
\textbf{Step 1.}
\textit{We have $B(x_1u) \,=\, x_1B(u)$.}

\smallskip
Indeed, by applying transformations of the form $v \mapsto v'=x_iv/x_j$ for $v \in B(x_1u)$ or $v \in x_1B(u)$, with $x_j$ dividing $v$ and $1 \le i < j$, the variable $x_1$ is not affected since $j \ge 2$. Whence the claimed equality.

\medskip
\noindent
\textbf{Step 2.}
\textit{We have $\widetilde{x_1u} \,=\, x_1\tilde{u}$.}

\smallskip
Indeed, it suffices to prove $L(\widetilde{x_1u}) \,=\, L(x_1\tilde{u})$. On the one hand, we have $L(\widetilde{x_1u})=B(x_1u)^{lex}$ by definition. Now $B(x_1u)=x_1B(u)$ by Step 1. Thus $L(\widetilde{x_1u})= (x_1B(u))^{lex}$, and $(x_1B(u))^{lex}=x_1B(u)^{lex}$ by Lemma~\ref{x_1B^lex}, and $x_1B(u)^{lex}=x_1L(\tilde{u})$ by definition. Finally, $x_1L(\tilde{u})= L(x_1\tilde{u})$ by Lemma~\ref{L(x_1u)} applied to $\tilde{u}$. This concludes the proof of Step 2.

\medskip
\noindent
\textbf{Step 3.}
\textit{For all $B \subseteq S_{n,d}$, we have $\maxgen(x_1B) \,=\,\maxgen(B)$.}

\smallskip
Indeed, this follows from Lemma~\ref{maxgen} and the obvious equality $\lambda(x_1v)=\lambda(v)$ for all $v \in S_{n,d}$.

\medskip
We may now compare the maxgen monomials of $\gaps(u), \cogaps(u)$ with those of $\gaps(x_1u), \cogaps(x_1u)$, respectively. First, by Proposition~\ref{gaps x_1u} and Step 3, we have
\begin{equation}\label{x_1u}
\maxgen(\gaps(x_1u)) \,=\, \maxgen(\gaps(u)).
\end{equation}
Symmetrically, we also have
\begin{equation}\label{cogaps x_1u}
\maxgen(\cogaps(x_1u)) \,=\, \maxgen(\cogaps(u)),
\end{equation}
as we now show:
\begin{eqnarray*}
\maxgen(\cogaps(x_1u)) & = & \maxgen(L(x_1u) \setminus L(\widetilde{x_1u}))\\
& = & \maxgen(L(x_1u) \setminus L(x_1\tilde{u}))  \quad \textrm{(by Step 2)}\\
& = & \maxgen(x_1(L(u) \setminus L(\tilde{u})))  \quad \textrm{(by Lemma~\ref{L(x_1u)})}\\
& = & \maxgen(L(u) \setminus L(\tilde{u}))  \quad \textrm{(by Step 3)}\\
& = & \maxgen(\cogaps(u)).
\end{eqnarray*}
The desired equivalence is now easy to establish. Indeed, it follows from \eqref{x_1u} and \eqref{cogaps x_1u} that
\begin{equation}\label{new  equality}
\maxgen(\gaps(x_1u)) \,=\, \maxgen(\cogaps(x_1u))
\end{equation}
if and only if
\begin{equation}\label{gotz}
\maxgen(\gaps(u)) \,=\, \maxgen(\cogaps(u)).
\end{equation}
Therefore, $x_1u$ is Gotzmann if and only if $u$ is Gotzmann.
\end{proof}

\subsection{On $\gaps(ux_n)$}

For use in the next section, we shall need to control $\maxgen(\gaps(ux_n))$.

\begin{definition} Let $u \in S_n$. For all $i \le n$, denote 
$$
\gaps(u, i) \,=\, \{v \in \gaps(u) \mid \max v = i\}.
$$
\end{definition}
Note that $\gaps(u, 1)$ is empty, for $x_1^d$ cannot be a gap since it obviously belongs to $B(u)$ for all $u \in S_{n,d}$.

\begin{theorem}\label{gaps ux_n} Let $u \in S_n$. Then for all $1 \le j \le n$, we have
\begin{equation}\label{gaps of x_n u}
\gaps(ux_n, j) \,=\, \bigsqcup_{i = 1}^j \,  \gaps(u, i)x_j.
\end{equation}
\end{theorem}
Here is an equivalent formulation.

\begin{theorem} Let $u \in S_n$. Then, for any $w \in S_n$, we have
$$
w \in \gaps(ux_n) \ \Longleftrightarrow \ \frac{w}{\lambda(w)} \in \gaps(u).
$$
\end{theorem}
\begin{proof} Let $d = \deg(u)$, and write $u=x_{i_1}\cdots x_{i_d}$ with $1 \le i_1 \le \dots \le i_d \le n$. We may assume $\deg(w)=d+1$, for otherwise $w$ cannot be a gap of $ux_n$. Set $\max(w)=m$. Let $v=w/\lambda(w)$, and write $v=x_{j_1}\cdots x_{j_d}$ with $1 \le j_1 \le \dots \le j_d \le m$. 

By Lemma~\ref{gaps}, $v$ is a gap of $u$ if and only if there exist indices $1 \le s < t \le d$ such that $j_s < i_s$ and $j_t > i_t$. If these conditions are met, then since $w=vx_m$ with $m \ge \max(v)$, then automatically $w$ is a gap of $ux_n$, still by Lemma~\ref{gaps}. Conversely, if $w$ is a gap of $ux_n$, and since $\max(w) \le \max(ux_n)=n$, then the index $t \le d+1$ given by Lemma~\ref{gaps} necessarily satisfies $t \le d$. Hence $v$ is a gap of $u$.
\end{proof}

\begin{corollary}\label{maxgen gaps of x_n u} Let $u \in S_n$. If 
$\displaystyle
\maxgen(\gaps(u)) \,=\, \prod_{i=1}^n x_i^{k_i}
$
then 
$$
\maxgen(\gaps(ux_n)) \,=\, \prod_{j=1}^n x_j^{k_1+\dots+k_j}.
$$
\end{corollary}
\begin{proof} By Theorem~\ref{gaps ux_n}, we have
$$
|\gaps(ux_n, j)| \,=\, \sum_{i = 1}^j \,  |\gaps(u, i)x_j|| \,=\, \sum_{i = 1}^j \,  |\gaps(u, i)|.
$$
The statement now follows from the definition of the maxgen monomial.
\end{proof}

\section{Gotzmann monomials in $S_2, S_3, S_4$}\label{sec gotzmann}

This section contains the main result of this paper, namely the characterization of Gotzmann monomials in $S_n$ for $n=4$. This is achieved in Theorem~\ref{main thm}. The strategy is as follows. Let $u = x_1^{a_1}\cdots x_{n-1}^{a_{n-1}}x_n^t \in S_n$. We may assume $a_1=0$ by Theorem~\ref{u and x_1u}, according to which $u$ is Gotzmann in $S_n$ if and only if $u/x_1^{a_1}$ is. We first compute $w_1=\maxgen(\gaps(u))$ using Theorem~\ref{thm maxgen gaps u}. The degree $g$ of $w_1$ gives the numbers of gaps of $u$. We then focus on $\cogaps(u)=\pred_g(u)$ and, more precisely, compute its maxgen monomial $w_2=\maxgen(\pred_g(u))$. Finally, requiring $w_1=w_2$ gives necessary and sufficient conditions on the exponent $t$ of $x_n$ for $u$ to be a Gotzmann monomial. 

Before turning to the case $n=4$, we start by reviewing the known cases $n=2$ and $3$. 

\subsection{The case $n=2$}
This is easy. Indeed, \textit{every monomial $u=x_1^ax_2^t$ is Gotzmann in $S_2$}. For in this case, the sets $B(u)$ and $L(u)$ coincide, whence $B(u)^{lex}=B(u)$ and so $B(u)$ is a Gotzmann set by Lemma~\ref{gotzmann sets}.

\subsection{The case $n=3$}

The result below for $n=3$ may be deduced from \cite[Proposition 8]{Mu2}. As an illustration of the strategy briefly described above, we give here an independent short proof using the tools developed in this paper.

\begin{proposition} Let $u=x_1^ax_2^bx_3^t \in S_3$. Then $u$ is a Gotzmann monomial in $S_3$ if and only if $t \ge \binom{b}{2}$.
\end{proposition}
\begin{proof} Let $g=|\gaps(u)|$, $w_1=\maxgen(\gaps(u)), w_2=\maxgen(\cogaps(u))$. Then $g=\deg(w_1)=\deg(w_2)$. A straightforward computation with Theorem~\ref{thm maxgen gaps u} yields the monomial
$$w_1 = x_3^{\binom{b}{2}}$$
independent of $t$. Therefore $g=\binom{b}{2}$. Thus $\cogaps(u)=\pred_g(u)$. Consider now $w_2=\maxgen(\pred_g(u))$. For all $i \le t$, we have $\pred^i(u)=x_1^ax_2^{b+i}x_3^{t-i}$. Thus $\lambda(\pred^i(u))=x_3$ if $i<t$ and $\lambda(\pred^t(u))=x_2$. Hence $\maxgen(\pred_t(u))=x_3^t$ and $\maxgen(\pred_{t+1}(u))=x_2x_3^t$. Consequently, if $t < g$ then $x_2$ divides $w_2$ by Remark~\ref{divides} and so $w_2 \not= w_1$, whereas if $t \ge g$ then $w_2=x_3^g=w_1$. Thus $u$ is Gotzmann if and only $t \ge g$, as claimed.
\end{proof}

\subsection{The case $n=4$}

Our purpose in this section is to determine all Gotzmann monomials in $4$ variables. This is achieved in Theorem~\ref{main thm}. As recalled above, it suffices to consider monomials of the form $x_2^bx_3^cx_4^t$. Implementing our proof strategy requires several preliminary results. 

\medskip
We start by determining $\maxgen(\gaps(x_2^bx_3^cx_4^t))$.

\begin{proposition}\label{prop maxgen gaps u} Let $u_0=x_2^bx_3^c \in S_4$. Then for all $t \ge 0$, we have
\begin{equation}\label{f(t)}
\maxgen(\gaps(u_0x_4^t)) \ = \ x_3^{\binom{b}{2}}x_4^{f(t)},
\end{equation}
where
\begin{eqnarray*}\label{f0}
f(t) & = & f(0)+t\binom{b}{2}, \\
f(0) & = & \big(\frac{b+1}{3}+c\big)\binom{b}{2}+(b+1)\binom{c+1}{2}+\binom{c+1}{3}-c.
\end{eqnarray*}
\end{proposition}
\begin{proof} \hfill

\noindent
\textbf{Case $t=0$.} This is the longest part of the proof, yet it follows almost mechanically from Theorem~\ref{thm maxgen gaps u} and a few formulas. In the notation of that result, let us write $u_0=x_{i_1}\cdots x_{i_d}$ with $i_1 \le \dots \le i_d$, where $d=\deg(u_0)=b+c$. Thus 
$$
u_0 =\underbrace{x_2\cdots x_2}_{b \textrm{ times}} \underbrace{x_3\cdots x_3}_{c \textrm{ times}}.
$$
Hence for $1 \le k \le d$, we have $i_k = 2$ if $k \le b$, and $i_k=3$ otherwise. By Theorem~\ref{thm maxgen gaps u}, we have
\begin{eqnarray*}
\maxgen(\gaps(u_0)) & = & \prod_{k=1}^{d-1}\left( \prod_{j=i_{k+1}+1}^n x_j^{\binom{d-k-2+j-i_{k+1}}{d-k-1}}\right)^{|B(x_{i_1}\cdots x_{i_k})|-1} \\
& = & \prod_{k=1}^{b-1}\left( \prod_{j=3}^4 x_j^{\binom{d-k-2+j-2}{d-k-1}}\right)^{|B(x_2^k)|-1} \cdot \prod_{k=b}^{d-1}\left( x_4^{1}\right)^{|B(x_2^bx_3^{k-b})|-1} \\
& = & \prod_{k=1}^{b-1}\left( x_3^1x_4^{d-k}\right)^{|B(x_2^k)|-1} \cdot \prod_{k=b}^{d-1}\left( x_4^{1}\right)^{|B(x_2^bx_3^{k-b})|-1}.
\end{eqnarray*}
We now compute the involved exponents. We have $|B(x_2^k)|=k+1$, as follows from the set equality $B(x_2^k)=\{x_1^{k-i}x_2^i \mid  0 \le i \le k\}$. On the other hand, we have
$$
|B(x_2^rx_3^s)| \ = \ \binom{s+1}{2}+(r+1)(s+1),
$$
as follows from the formula
$$
|B(x_2^rx_3^s)| \ = \ \sum_{i=0}^s |B(x_{2}^{r+s-i})|
$$
of Corollary~\ref{cor card <w>}, 
the above formula for $|B(x_2^k)|$ and some straightforward computations.

\smallskip
Inserting these exponent values into the above formula for $\maxgen(\gaps(u_0))$, we get
\begin{eqnarray*}
\maxgen(\gaps(u_0)) & = & \prod_{k=1}^{b-1}\left( x_3x_4^{d-k}\right)^{k} \cdot \prod_{k=b}^{d-1}x_4^{\binom{k-b+1}{2}+(b+1)(k-b+1)-1} \\
  & = & x_3^{\binom{b}{2}} x_4^{A+B},
\end{eqnarray*}
where 
\begin{eqnarray*}
A & = & \sum_{k=1}^{b-1}k(d-k), \\
B & = & \sum_{k=b}^{d-1}\big(\binom{k-b+1}{2}+(b+1)(k-b+1)-1\big).
\end{eqnarray*}
By the formula
$$
\sum_{k=1}^{b-1} k ^2 \ = \ \frac{2b-1}{3}\binom{b}{2}
$$
and some straightforward computations, we get
$$A = \big(\frac{b+1}{3}+c\big)\binom{b}{2}.$$
Similarly, the formula
$$
\sum_{l=1}^c \binom{l}{2} = \binom{c+1}{3}
$$
and some further straightforward computations yield
$$
B = (b+1)\binom{c+1}{2}+\binom{c+1}{3}-c.
$$
As $f(0)=A+B$, the proof of formula~\eqref{f(t)} in case $t=0$ is complete.

\medskip
\noindent
\textbf{Case $t \ge 1$.} For all $s \ge 1$, Corollary~\ref{maxgen gaps of x_n u} and the above case $t=0$ imply
$$
\maxgen(\gaps(u_0x_4^s)) \ = \ \maxgen(\gaps(u_0x_4^{s-1}))x_4^{\binom{b}{2}}
$$
by induction on $s$. The claimed formula 
$$
\maxgen(\gaps(u_0x_4^t)) \ = \ x_3^{\binom{b}{2}}x_4^{f(t)}
$$
follows by induction on $t$.
\end{proof}

\bigskip

We now proceed to determine $\maxgen(\cogaps(x_2^bx_3^cx_4^t))$. We first need two lemmas.

\begin{lemma}\label{i=1} For all $r \ge 0, s \ge 1$, we have
$$
\mu(x_2^rx_4^s, x_2^{r+1}x_4^{s-1}) = x_3 x_4^{s}.
$$
\end{lemma}
\begin{proof} Starting from $x_2^rx_4^s$ and taking $s+1$ successive predecessors, Proposition~\ref{predecessor} yields
$$
x_2^rx_4^s  \overset{x_4^s}{\xrightarrow{\hspace*{1.8cm}}}  x_2^bx_3^{s}  \overset{x_3}{\xrightarrow{\hspace*{1.8cm}}} x_2^{r+1}x_4^{s-1}
$$
in arrow notation, \emph{i.e.} $\mu(x_2^rx_4^s, x_2^bx_3^{s}) = x_4^{s}$ and $\mu(x_2^bx_3^{s}, x_2^{r+1}x_4^{s-1}) = x_3$. The desired formula follows by arrow composition.
\end{proof}

\begin{lemma}\label{any i} For all $r \ge 0$ and $1 \le i \le s$, we have
$$
\mu(x_2^rx_4^s, x_2^{r+i}x_4^{s-i}) = x_3^i x_4^{\binom{i+1}{2}+i(s-i)}.
$$
\end{lemma}
\begin{proof} By induction on $i$. The case $i=1$ is just Lemma~\ref{i=1}. By arrow composition, we have
$$
\mu(x_2^rx_4^s, x_2^{r+i}x_4^{s-i}) = \prod_{j=0}^{i-1} \mu(x_2^{r+j}x_4^{s-j}, x_2^{r+j+1}x_4^{s-j-1}).
$$
Applying Lemma~\ref{i=1} again to each factor, we get
$$
\mu(x_2^rx_4^s, x_2^{r+i}x_4^{s-i}) = x_3^ix_4^{\sum_{j=0}^{i-1} (s-j)}.
$$
Finally, $\sum_{j=0}^{i-1} (s-j)= \binom{i+1}{2}+i(s-i)$ and the proof is complete.
\end{proof}

\begin{proposition}\label{h(t)} We have
$$
x_2^bx_3^cx_4^t  \overset{x_3^{\binom{b}{2}}x_4^{h(t)}}{\xrightarrow{\hspace*{2cm}}}x_2^{b+\binom{b}{2}}x_4^{c+t-\binom{b}{2}},
$$
where 
$$
h(t) \, = \, (c+t)\binom{b}{2}-\binom{\binom{b}{2}}{2}-c.
$$
\end{proposition}
\begin{proof}
Starting from $x_2^bx_3^cx_4^t$ and taking $t+1$ successive predecessors, Proposition~\ref{predecessor} yields
$$
x_2^bx_3^cx_4^t \overset{x_4^t}{\xrightarrow{\hspace*{1.8cm}}} x_2^bx_3^{c+t}  \overset{x_3}{\xrightarrow{\hspace*{1.8cm}}} x_2^{b+1}x_4^{c+t-1}.
$$
Hence
\begin{equation}\label{initial step}
\mu(x_2^bx_3^cx_4^t,x_2^{b+1}x_4^{c+t-1})=x_3x_4^t.
\end{equation}
From $x_2^{b+1}x_4^{c+t-1}$, we must still reach $x_2^{b+\binom{b}{2}}x_4^{c+t-\binom{b}{2}})$. This can be done using Lemma~\ref{any i} to $(r,s,i)=(b+1,c+t-1,\binom{b}{2}-1)$. We obtain
\begin{equation}\label{general step}
\mu(x_2^{b+1}x_4^{c+t-1}, x_2^{b+\binom{b}{2}}x_4^{c+t-\binom{b}{2}})=
x_3^{\binom{b}{2}-1} x_4^{\binom{\binom{b}{2}}{2}+(\binom{b}{2}-1)(c+t-\binom{b}{2})}.
\end{equation}
Hence, combining \eqref{initial step} and \eqref{general step} using arrow composition, we get
$$
\mu(x_2^bx_3^cx_4^t,x_2^{b+\binom{b}{2}}x_4^{c+t-\binom{b}{2}}) = 
x_3^{\binom{b}{2}} x_4^{\binom{\binom{b}{2}}{2}+(\binom{b}{2}-1)(c+t-\binom{b}{2})+t}.
$$
It remains to show that the exponent of $x_4$ in the monomial of the above right-hand side is equal to $h(t)$. Indeed, we have
\begin{eqnarray*}
\binom{\binom{b}{2}}{2}+(\binom{b}{2}-1)(c+t-\binom{b}{2})+t & = & \\
\binom{\binom{b}{2}}{2}+(c+t)\binom{b}{2} -(\binom{b}{2}-1)\binom{b}{2}-c.
\end{eqnarray*}
Since
$$
-(\binom{b}{2}-1)\binom{b}{2} = -2\binom{\binom{b}{2}}{2},
$$
the desired equality with $h(t)$ follows.
\end{proof}

\begin{remark} By Propositions \ref{prop maxgen gaps u} and \ref{h(t)}, for $t \ge 0$ we have
\begin{eqnarray*}
f(t)-h(t) & = & f(0) - c\binom{b}{2}+\binom{\binom{b}{2}}{2}+c\\
& = & \frac{b+1}{3}\binom{b}{2}+(b+1)\binom{c+1}{2}+\binom{c+1}{3}+\binom{\binom{b}{2}}{2}.
\end{eqnarray*}
In particular, $f(t)-h(t)$ is a positive constant. This will be used below.
\end{remark}

Here is our main result.
\begin{theorem}\label{main thm}
Let $u=x_1^ax_2^bx_3^cx_4^t \in S_4$. Then $u$ is a Gotzmann monomial in $S_4$ if and only if
$$
t \ \ge \ \binom{\binom{b}{2}}{2}+\frac{b+4}{3}\binom{b}{2}+(b+1)\binom{c+1}{2}+\binom{c+1}{3}-c.
$$
\end{theorem}
As expected, the absence of exponent $a$ in this bound on $t$ is consistent with Theorem~\ref{u and x_1u}. 
\begin{proof} By Theorem~\ref{u and x_1u}, we may assume $a=0$. Denote $u_0=x_2^bx_3^c$, so that $u=u_0x_4^t$. There are two steps. 

\smallskip
\noindent
\textbf{Step 1.} \textit{The monomial $u_0x_4^t$ is Gotzmann if and only if}
\begin{equation}\label{condition}
t \ \ge \ f(t)-h(t)+\binom{b}{2}-c.
\end{equation}

Indeed, by Proposition~\ref{prop maxgen gaps u}, we have
\begin{equation}\label{maxgen gaps u}
\maxgen(\gaps(u_0x_4^t))=x_3^{\binom{b}{2}}x_4^{f(t)}.
\end{equation}
Thus $|\gaps(u_0x_4^t)|=\binom{b}{2}+f(t)$. For $u_0x_4^t$ to be a Gotzmann monomial, we apply the criterion given by Theorem~\ref{Gotzmann monomials}. Thus, by \eqref{maxgen gaps u}, we need to determine those $t \ge 0$ for which 
\begin{equation}\label{criterion}
\maxgen(\cogaps(u_0x_4^t))=x_3^{\binom{b}{2}}x_4^{f(t)}.
\end{equation}
Now $\cogaps(u_0x_4^t)=\pred_{\binom{b}{2}+f(t)}(u_0x_4^t)$ by Proposition~\ref{pred and cogaps}. 
In order to compute the maxgen monomial of the set of $\binom{b}{2}+f(t)$ predecessors of $u=u_0x_4^t$, we first compute it for its $\binom{b}{2}+h(t)$ predecessors. Let
\begin{eqnarray*}
LI(t) & = & \pred_{\binom{b}{2}+h(t)}(u_0x_4^t), \\
v(t) & = & \pred^{\binom{b}{2}+h(t)}(u_0x_4^t).
\end{eqnarray*}
Then $LI(t)=L^*(v(t), u_0x_4^t)$, and 
we seek the maxgen monomial of this lexinterval. By Proposition~\ref{h(t)}, we have
\begin{equation}\label{v(t)}
u_0x_4^t=x_2^bx_3^cx_4^t  \overset{x_3^{\binom{b}{2}}x_4^{h(t)}}{\xrightarrow{\hspace*{2cm}}}x_2^{b+\binom{b}{2}}x_4^{c+t-\binom{b}{2}}.
\end{equation}
Hence
\begin{eqnarray*}
v(t) & = & x_2^{b+\binom{b}{2}}x_4^{c+t-\binom{b}{2}}, \\
\maxgen(LI(t)) & = & x_3^{\binom{b}{2}}x_4^{h(t)}.
\end{eqnarray*}
Now, restarting from $v(t)$, it remains to compute $f(t)-h(t)$ more predecessors in order to reach $\pred_{\binom{b}{2}+f(t)}(u_0x_4^t)$. We'll then have
\begin{eqnarray*}
\maxgen(\cogaps(u_0x_4^t)) & = & \maxgen(LI(t))\maxgen(\pred_{f(t)-h(t)}(v(t))) \\
& = & x_3^{\binom{b}{2}}x_4^{h(t)}\maxgen(\pred_{f(t)-h(t)}(v(t))).
\end{eqnarray*}
Therefore, in order to satisfy equality~\eqref{criterion} for $u_0x_4^t$ to be a Gotzmann monomial, it is necessary and sufficient to satisfy
$$
\maxgen(\pred_{f(t)-h(t)}(v(t)))=x_4^{f(t)-h(t)}.
$$
Since $v(t)=x_2^{b+\binom{b}{2}}x_4^{c+t-\binom{b}{2}}$, the above condition is realizable if and only if the exponent of $x_4$ in $v(t)$ is large enough, namely satisfies
$$
c+t-\binom{b}{2} \ge f(t)-h(t).
$$
This condition being equivalent to $\eqref{condition}$, the proof of the claim in Step 1 is complete.

\smallskip
\noindent
\textbf{Step 2.} A straightforward computation on the right-hand side of \eqref{condition} yields
{\small
\begin{eqnarray*}
f(t)-h(t)+\binom{b}{2}-c & = & f_0+t\binom{b}{2}-\big((c+t)\binom{b}{2}-\binom{\binom{b}{2}}{2}-c\big)+\binom{b}{2}-c \\
& = & f_0-c\binom{b}{2}+\binom{\binom{b}{2}}{2}+c+\binom{b}{2}-c \\
& = & f_0-c\binom{b}{2}+\binom{\binom{b}{2}}{2}+\binom{b}{2} \\
& = & \binom{\binom{b}{2}}{2}+\frac{b+4}{3}\binom{b}{2}+(b+1)\binom{c+1}{2}+\binom{c+1}{3}-c.
\end{eqnarray*}}
The conjunction of Steps 1 and 2 completes the proof of the theorem.
\end{proof}

\bigskip
\bigskip

\noindent
{\small
\textbf{Authors addresses}

\medskip

\noindent
$\bullet$ Vittoria {\sc Bonanzinga}\textsuperscript{a,b},

\noindent
\textsuperscript{a} Univ. Mediterranea di Reggio Calabria, 
DIIES\\
\textsuperscript{b} Via Graziella (Feo di Vito), 89100 Reggio Calabria, Italia \\
\email{vittoria.bonanzinga@unirc.it}

\medskip

\noindent
$\bullet$ Shalom {\sc Eliahou}\textsuperscript{a,b},

\noindent
\textsuperscript{a}Univ. Littoral C\^ote d'Opale, EA 2597 - LMPA - Laboratoire de Math\'ematiques Pures et Appliqu\'ees Joseph Liouville, F-62228 Calais, France\\
\textsuperscript{b}CNRS, FR 2956, France\\
\email{eliahou@univ-littoral.fr}
}

\end{document}